\documentclass[11pt,reqno]{amsart}

\usepackage{tikz}
\usepackage{aliascnt}
\usepackage{amsfonts}
\usepackage{mathrsfs}
\usepackage{amsthm}
\usepackage{amsmath,amssymb}
\usepackage[all,poly,knot]{xy}
\usepackage{amsfonts}
\usepackage{color}
\usepackage{hyperref}
\usepackage{comment}
\usepackage{csquotes}

\usepackage{aliascnt}
\usepackage{hyperref}
\usepackage{amsfonts}
\usepackage{mathrsfs}
\usepackage{amsmath}
\usepackage{fullpage}
\usepackage{amsmath,graphics}
\usepackage{amssymb}
\usepackage{indentfirst,latexsym,bm,amsmath,pstricks,amssymb,amsthm,graphicx,fancyhdr,float,color}
\usepackage[all]{xy}
\usepackage{amsthm}
\usepackage{amscd}
\usepackage{enumerate}
\usepackage{enumitem}
\usepackage[all]{hypcap}

\def\lra{\longrightarrow}
\theoremstyle{plain}
\newtheorem{thmx}{Theorem}

\newtheorem{thm}{Theorem}[section]

\newtheorem{lem}[thm]{{Lemma}}

\newtheorem{prop}[thm]{Proposition}

\newtheorem{defi}[thm]{Definition}

\theoremstyle{remark}

\numberwithin{equation}{section}

\def\calo{\mathcal{O}}

\def\cald{\mathcal D}

\def\calm{\mathcal M}

\def\calo{\mathcal O}

\def\calx{\mathcal X}
\def\calw{\mathcal W}

\usepackage[top=1.5in, bottom=1.5in, left=1.2in, right=1.2in]{geometry}

 \usepackage[hyperpageref]{backref}
 
\usepackage{xcolor}
\hypersetup{
	colorlinks,
	linkcolor={red!50!black},
	citecolor={blue!50!black},
	urlcolor={blue!80!black}
}
\begin{document}

\title[Topological Hyperbolicity of Moduli Spaces]{Topological Hyperbolicity of Moduli Spaces of Elliptic Surfaces}


\author{Xin L\"u}
\address{School of Mathematical Sciences, Key Laboratory of MEA(Ministry of Education) \& Shanghai Key Laboratory of PMMP, East China Normal University, Shanghai 200241, China}
\email{xlv@math.ecnu.edu.cn}
\thanks{This work is supported by Shanghai Pilot Program for Basic Research, National National Science foundation of China, and Science and Technology Commission of Shanghai Municipality (No. 22DZ2229014).}

\author{Ruiran Sun}
\address{Department of Mathematics and Statistics, McGill University, Montr\'{e}al, Qu\'{e}bec, Canada, H3A 0B9}
\email{sunruiran@gmail.com}

\author{Kang Zuo}
\address{ School of Mathematics and Statistics, Wuhan University, Luojiashan, Wuchang, Wuhan, Hubei, 430072, P.R. China; Institut f\"ur Mathematik, Universit\"at Mainz, Mainz, Germany, 55099}
\email{zuok@uni-mainz.de}

\phantomsection
\begin{abstract}
\addcontentsline{toc}{section}{Abstract}
We introduce the notion of topological hyperbolicity to characterize the largeness of the topological fundamental group of a complex variety. Inspired by the Shafarevich conjecture, we propose to study the topological hyperbolicity of moduli spaces of polarized manifolds. We provide two pieces of supporting evidence: first, we show that moduli spaces where the infinitesimal Torelli theorem holds are very close to being topologically hyperbolic. Second, we establish a weak form of topological hyperbolicity for moduli spaces of elliptic surfaces of Kodaira dimension one without multiple fibers, where the infinitesimal Torelli theorem generally does not hold.
\end{abstract}

\maketitle



\section{Introduction}\label{introduction}
At the 1962 International Congress of Mathematicians in Stockholm, Shafarevich formulated an influential conjecture, now known as the Shafarevich conjecture. One part of the conjecture is the following
\begin{thm}[Shafarevich Hyperbolicity Conjecture]
Let $Y$ be a smooth projective curve of genus $g$ over an algebraically closed field $k$ of characteristic $0$ and $\Delta \subset Y$ a finite subset. Fix an integer $q \geq 2$. Let $f:\, X \to Y$ be a non-isotrivial family of curves of genus $q$, and the resitiction $f:\, X \setminus f^{-1}(\Delta) \to Y \setminus \Delta$ is smooth. Then
\[
2g-2+ \# \Delta >0.
\]
\end{thm}

Here the non-isotriviality of $f$ means that the family $f$ cannot become trivial after any finite base change.
The above Shafarevich hyperbolicity conjecture was proved by Parshin \cite{par-68} and Arakelov \cite{ara-71}. By the theorem above we know that the quasi-projective curve $U:= Y \setminus \Delta$ is a hyperbolic curve. Over the field of complex numbers, an algebraic curve $U$ being hyperbolic is characterized in three equivalent ways:
\begin{enumerate}
	\item the Euler characteristic of $U$ is negative;
	\item there exists  a negatively curved metric on $U$;
	\item the fundamental group $\pi_1(U)$ is infinite and non-abelian.
\end{enumerate}

In consideration of the importance of the Shafarevich conjecture in algebraic and arithmetic geometry, it is significant to generalize the above three perspectives to the higher dimension cases.
Let $f:\, V \to U$ be a family of polarized manifolds with semi-ample canonical divisors. One expects the base space $U$ to be ``hyperbolic''.
When $U$ is a quasi-projective variety of higher dimension, there are different notions of hyperbolicity;
for instance, the Brody hyperbolicity and Kobayashi hyperbolicy (cf. \cite{kob-98}).
In this paper, we would like to introduce the notion of \emph{topological hyperbolicity}. We first recall that for a finitely generated group one can associate a growth function $\ell$: for each positive integer $s$ let $\ell(s)$ be the number of distinct group elements which can be expressed as words of length $\leq s$ with a fixed choice of generators and their inverse (cf. for instance \cite{Mil68}).
We say the finitely generated group has \emph{exponential growth} if $\ell(s) \approx a^s$ for a real constant $a>1$. The notion of {exponential growth} characterizes the large size of this group. By definition one can check that if a finitely generated group has exponential growth, then it has to be infinite and non-abelian.
\begin{defi}
A quasi-projective variety $U$ over $\mathbb{C}$ is said to be \emph{topologically hyperbolic} if its topological fundamental group satisfies the following property: for each subvariety $Z \subset U$ of positive dimension, $\pi_1(\widetilde{Z})$ has exponential growth (where $\widetilde{Z}$ is the normalization of $Z$), and the image  $\mathrm{Im}[\pi_1(\widetilde{Z}) \to \pi_1(U)]$ is infinite and non-abelian.
\end{defi}

All these notions of hyperbolicity should be viewed as some sort of high dimensional generalization of the three equivalent characterization of the hyperbolic curves.
It is well-known that the Kobayashi hyperbolicity implies the Brody hyperbolicity; and these two notions are equivalent for compact complex manifolds.
By a conjecture of Lang (cf. \cite{L86}), a projective variety is Brody (equivalently Kobayashi) hyperbolic if and only if every subvariety of positive dimension is of general type.
However, the notion of topological hyperbolicity seems to be {\it stronger}.
In fact, there exist Kobayashi hyperbolic manifolds which are not topologically hyperbolic:
for instance, the complement $\mathbb{P}^N \setminus H$ is Kobayashi hyperbolic (cf. \cite{SY96}) for a general smooth hypersurface $H\subseteq \mathbb{P}^N$ of sufficiently large degree,
but its fundamental group $\pi_1(\mathbb{P}^N \setminus H)$ is abelian ( cf. {\cite[Corollary~2.8]{nori-83}} ) and hence not topologically hyperbolic.
Assuming Lang's conjecture, one sees that if the variety is projective of dimension two, then the topological hyperbolicity implies the Kobayashi hyperbolicity and hence also the Brody hyperbolicity.
However, the converse does not hold in general:
for example, a general hypersurface $H\subseteq \mathbb{P}^N$ of sufficiently high degree is Kobayashi hyperbolic (cf. \cite{Siu02}),
but on the other hand if $N\geq 3$, $H$ is simply connected by the Lefschetz hyperplane theorem and hence not topologically hyperbolic.

In the past decades, many works on the hyperbolicity properties of the moduli space of polarized varieties
have been invested, motivated by the Shafarevich problem and its higher dimensional generalizations.
For instance, Viehweg and the third named author proved that $U$ parametrizing canonically polarized manifolds 
is 
Brody hyperbolic in \cite{vz}.
Later, To-Yeung have made a further step by showing that $U$ is Kobayashi hyperbolic \cite{ToYeung}.
For more details and references, the readers are recommended to the surveys \cite{vie-01}, \cite{kovac-03}, \cite{mvz-05}, \cite{kovac-05} etc.
However, very little is known for the topological hyperbolicity.

In this paper we propose to study the topological hyperbolicity of moduli spaces of polarized manifolds. As the first supporting evidence, we prove that:
\begin{thmx}\label{main-thm1}
  Let $X$ be a projective manifold supporting a polarized integral variation of Hodge structures such that the induced period map is quasi-finite. Then for each smooth closed subvariety $Z \subset X$ of positive dimension, $\pi_1(Z)$ has exponential growth and the image $\mathrm{Im}[\pi_1(Z) \to \pi_1(X)]$ is infinite and non-abelian.
\end{thmx}
Theorem~\ref{main-thm1} suggests that moduli spaces where Torelli-type theorems hold tend to be topologically hyperbolic.

If a finitely generated group $\Gamma$ is linear (i.e. can be realized as a subgroup of some general linear group), a classical result tells us that $\Gamma$ is either of polynomial growth in which case $\Gamma$ is virtually nilpotent, or of exponential growth otherwise (\cite{Tit72, Mil68b, Wol68}). To derive Theorem~\ref{main-thm1} as a corollary, it is essential to rule out the possibility of the monodromy representations of polarized VHS being virtually nilpotent. However, Deligne's classical semi-simplicity theorem only imposes restrictions on the algebraic monodromy groups (i.e., the Zariski closure of the monodromy representation image). Thus one has limited information about the monodromy group itself, which is a finitely generated linear group, unless the monodromy group can be explicitly determined in certain cases (such as the mapping class group, cf. \cite{AAS}). Consequently, rather than analyzing the monodromy representation directly, we adopt Milnor's strategy from \cite{Mil68} to prove the exponential growth property.

The proof of Theorem~\ref{main-thm1} utilizes a crucial volume estimate from \cite{BT-Ax} and Milnor's method from \cite{Mil68}. Given that the image of the quasi-finite period map may have singularities, we generalize Milnor's method to accommodate our setting.

Conversely, it is natural to investigate moduli spaces where Torelli-type theorems are {\it invalid}. As the second result, we prove that the base space of family of elliptic surfaces of Kodaira dimension one satisfies a weak form of topological hyperbolicity:
\begin{thmx}\label{main-thm2}
  Let $f:\,\calx \to B$ be a smooth family of polarized elliptic surfaces of Kodaira dimension one. Assume that every fiber of $f$ is a relative minimal elliptic surface
 without multiple fiber.
	Assume that the induced classifying map from the base $B$ to the coarse moduli space is quasi-finite. Then for each subvariety $Z \subset B$ of positive dimension ($Z$ can be equal to $B$), the fundamental group $\pi_1(\widetilde{Z})$ is infinite and non-abelian, where $\widetilde{Z}$ is the normalization of $Z$.
\end{thmx}
We expect the base $B$ that appears in \autoref{main-thm2} to be topologically hyperbolic.
We hope that Theorem~\ref{main-thm2} serves as a first step toward studying the topological hyperbolicity of more general moduli spaces.
It is well-known that the Torelli theorem fails for
elliptic surfaces of Kodaira dimension one \cite{chakiris-80}.
In \cite{saito-83}, Saito provided a sufficient condition for the validity of the infinitesimal Torelli theorem for such elliptic surfaces. This significant contribution allows us to focus on the remaining cases where Saito's condition does not apply. In these instances, we will utilize Oguiso-Viehweg's elegant work on the analysis of the singular locus of elliptic surfaces (cf. \cite{ov-01}), along with Simpson's nonabelian Hodge theory, to complete the proof.\\[.3cm]

\noindent{\bf Acknowledgment.} Part of this paper was written during the visits of the second and third named authors to the School of Mathematical Sciences of ECNU in August 2018. They thank all the members there for their hospitality.




\section{Variation of Hodge structures and topological hyperbolicity}
In this section we prove Theorem~\ref{main-thm1}. We first recall the following well-known fact
\begin{prop}\label{prop-1-1}
  Let $X$ be a quasi-projetive variety carrying a polarized variation of Hodge structures $\mathbb{V}$. If the induced period map is nonconstant on $X$, then the image of the monodromy representation $\pi_1(X) \to \Gamma$ associated with $\mathbb{V}$ is an infinite and non-abelian group.\\
  In particular, $\pi_1(X)$ is infinite and non-abelian.
\end{prop}
One can find in {\cite[Theorem~3.1]{Voi_BOOK}} and {\cite[\S3, p.153]{Zuo96}} for a proof of the above proposition. For each subvariety $Z \subset X$ of positive dimension, the induced monodromy representation $\pi_1(\widetilde{Z}) \to \Gamma_Z$ associated with $\mathbb{V}|_Z$ has infinite and non-abelian image, provided that the induced period map of $\mathbb{V}$ is quasi-finite. Since
\[
\mathrm{Im}[\pi_1(\widetilde{Z}) \to \pi_1(X)] \twoheadrightarrow \mathrm{Im}[\pi_1(\widetilde{Z}) \to \Gamma_Z],
\]
we know that $\mathrm{Im}[\pi_1(\widetilde{Z}) \to \pi_1(X)]$, and consequently $\pi_1(\widetilde{Z})$ itself, is infinite and non-abelian. Therefore, a weak form of topological hyperbolicity (infinite and non-abelian $\pi_1(\widetilde{Z})$ for each $Z \subset X$) holds if the period map is quasi-finite.

To study the exponential growth property of fundamental groups, one needs a volume estimate of Hodge metrics, which was proved by Bakker-Tsimerman:
\begin{thm}[Theorem~1.2 in \cite{BT-Ax}]\label{BT-vol}
  There are constants $\beta, \rho >0$ only depending on the period domain $\mathcal{D}$ of $\mathbb{V}$ such that for any $r > \rho$, any $x \in \mathcal{D}$, and any positive-dimensional Griffiths transverse closed analytic subvariety $Z \subset B_x(r) \subset \mathcal{D}$, we have
  \[
  \mathrm{vol}(Z) \geq \mathrm{exp}(\beta r) \mathrm{mult}_xZ
\]
where $B_x(r)$ is the radius $r$ geodesic ball centered at $x$ and $\mathrm{vol}(Z)$ the volume with respect to the natural left-invariant metric on $\mathcal{D}$.
\end{thm}

\begin{proof}[Proof of \autoref{main-thm1}]
By virtue of Proposition~\ref{prop-1-1}, we only need to show the exponential growth of $\pi_1(Z)$, where $Z \subset X$ is a smooth connected closed subvariety of positive dimension.

If $\mathrm{dim}\,Z=1$, then $Z$ is a compact Riemann surface with genus $\geq 2$. The classical uniformization theorem equips $Z$ with a Riemannian metric with constant negative sectional curvature. Then by Milnor's theorem {\cite[Theorem~2]{Mil68}}, we know that $\pi_1(Z)$ has exponential growth. Thus one only needs to prove the case that $\mathrm{dim}\,Z \geq 2$. Using the Lefschetz hyperplane theorem, the general case can be reduced to the surface case. Hence we will assume $\mathrm{dim}\,Z=2$ in the rest of the proof.

Consider the restriction of the period map on $Z$:
\[
\varphi:\,Z \to Y \subset \Gamma \backslash \mathcal{D}.
\]
Since $\varphi$ is quasi-finite, we know that $Y$ is a complete normal surface with finitely many quotient singularities.

\begin{lem}
If $\pi_1(Y)$ has exponential growth, then $\pi_1(Z)$ has exponential growth.
\end{lem}
\begin{proof}
  Let $\hat{Y} \to Y$ be the desingularization of $Y$. Consider the following Cartesian square
  \[
    \xymatrix{
      \hat{Z} \ar[r]^{\hat{\varphi}} \ar[d] & \hat{Y} \ar[d] \\
      Z \ar[r]^{\varphi} & Y.
}
  \]
Since $Z$ is smooth and $Y$ has only quotient singularities, one has $\pi_1(\hat{Z}) = \pi_1(Z)$ and $\pi_1(\hat{Y}) = \pi_1(Y)$. On the other hand, $\hat{\varphi}$ is a surjective morphism between two smooth connnected varieties. Thus by {\cite[Lemma~1.5, B]{nori-83}} $\pi_1(\hat{Z}) = \pi_1(Z)$ is a finite index subgroup of $\pi_1(\hat{Y}) = \pi_1(Y)$, which finishes the proof.
\end{proof}

Hence we aim to show that $\pi_1(Y)$ has exponential growth. This will follow from the volume estimate of Bakker-Tsimerman (Theorem~\ref{BT-vol}) and some argument similar as Milnor's in \cite{Mil68}.

Denote by $Y_{\mathcal{D}}$ some irreducible component of the fiber product $Y \times_{\Gamma \backslash \mathcal{D}} \mathcal{D}$. Then $Y_{\mathcal{D}}$ is a closed analytic subvariety in the period domain $\mathcal{D}$. Moreover, $p:\,Y_{\mathcal{D}} \to Y$ is an unramified cover. That induces the following commutative diagram
\[
\xymatrix{
  \tilde{Y} \ar[r] \ar[rd] & Y_{\mathcal{D}} \ar[d]^{p} \ar@{^{(}->}[r] & \mathcal{D} \ar[d] \\
   & Y \ar@{^{(}->}[r] & \Gamma \backslash \mathcal{D}
  }
\]
where $\tilde{Y}$ is the universal cover of $Y$. Thus the image of the monodromy representation
\[
G:= \mathrm{Im}[\pi_1(Y) \to \Gamma]
\]
is exactly the group of deck transformations of $p:\,Y_{\mathcal{D}} \to Y$. We will prove that $G$ has exponential growth, and thus so does $\pi_1(Y)$.

Let $\Sigma=\{y_1,\dots,y_n\}$ be the finite set of quotient singularities of $Y$. The restriction of the Hodge metric (deduced from the natural left-invariant metric on $\mathcal{D}$) on $Y$ makes $Y^o:= Y \setminus \Sigma$ a K\"ahler submanifold in $\Gamma \backslash \mathcal{D}$. It is clear that the compactification $Y$ is a metric space with finite diameter.

Choose a base point $x \in Y_{\mathcal{D}}$ outside the discrete subset $p^{-1}(\Sigma)$. For each $r>0$, one can consider the neighbourhood
\[
N_x(r):= \{y \in Y_{\mathcal{D}}\,:\,d_h(x,y) \leq r\}
\]
where $d_h(\cdot,\cdot)$ is the distance function associated with the left-invariant metric $h$ on $\mathcal{D}$. Then $N_x(r) = Y_{\mathcal{D}} \cap B_x(r) \subset B_x(r)$. From Theorem~\ref{BT-vol}, one obtains
\begin{equation}
  \label{vol_ineq}
V(r) := \mathrm{vol}(N_x(r)) \geq \mathrm{exp}(\beta r) \mathrm{mult}_xN_x(r) = \mathrm{exp}(\beta r) \mathrm{mult}_xY_{\mathcal{D}}
\end{equation}
for $r$ sufficiently large.

Let $\delta$ be the diameter of $Y$. Then the cover map $p:\,Y_{\mathcal{D}} \to Y$
maps the neighbourhood $N:=N_x(\delta)$ onto $Y$. Hence the set of all translates
\[
\{gN\,:\,g \in G\}
\]
forms a covering of $Y_{\mathcal{D}}$.

We claim that there are only finitely many $g \in G$ such that $gN$ intersects with $N$. For otherwise, $N_x(2\delta + \varepsilon)$ would contain infinitely many $gN$, and consequently, $N_x(2\delta + \varepsilon)$ would contain infinitely many disjoint set $gN_x(\varepsilon)$ for $\varepsilon$ sufficiently small. This is impossible since $N_x(2\delta + \varepsilon)$ has finite volume.

Now let $F:=\{g \in G\,:\,gN \cap N \neq \varnothing\}$, which is a finite set by the above argument. Let $\nu$ be the minimum of $d_h(gN,N)$ as $g$ ranges over $G \setminus F$. Since $G \subset \Gamma$ is a discrete subgroup, we know that $\nu>0$.

\begin{lem}[Lemma~2 in \cite{Mil68}]\label{word-est}
If $d_h(x,gN) < \nu t + \delta$ for some positive integer $t$, then $g$ can be expressed as a $t$-fold product, $g= f_1f_2\cdots f_t$ with $f_1,\cdots,f_t \in F$.
\end{lem}
The proof of Lemma~2 in \cite{Mil68} relies solely on the distance function and definitions of $F$ and $\nu$, and can be applied directly to our setting without any modifications. Therefore, we omit the proof of the above lemma.

Lemma~\ref{word-est} implies that elements in $F$ generate the group $G$. Let $\ell(t)$ be the growth function of $G$ using $F$ as the set of generators.

According to Lemma~\ref{word-est}, the interior of $N_{\nu t + \delta}(x)$ is completely covered by the translates $f_1f_2\cdots f_tN$ with $f_i \in F$. Note that $\ell(t)$ is by definition the number of distinct group elements which can be expressed as such $t$-fold products. Therefore,
\begin{equation}
  \label{word-vol-est}
\ell(t)V(\delta) \geq V(\nu t + \delta).
\end{equation}

Combining \eqref{vol_ineq} and \eqref{word-vol-est}, we obtain the inequality
\[
\ell(t) \geq \frac{e^{\beta \delta}\mathrm{mult}_xY_{\mathcal{D}}}{V(\delta)} \mathrm{exp}(\beta \nu t),
\]
which confirms the exponential growth of the group $G$. This completes the proof.
\end{proof}

\section{Topological hyperbolicity of moduli spaces of elliptic surfaces}
\subsection{Infinitesimal Torelli Theorem}\label{infinitesimal Torelli}
The (infinitesimal) Torelli problems for elliptic surfaces have been extensively studied, see for instance \cite{Ki78,saito-83,Kl04}.
Unfortunately, the infinitesimal Torelli does not hold for all elliptic surfaces.
Nevertheless, we still have the infinitesimal Torelli for certain families of elliptic surfaces.
For the readers' convenience, we recall the next theorem and sketch its proof following Saito's idea \cite{saito-83}.
We also refer to \cite{bhpv} for the general facts about elliptic surfaces.

\begin{thm}[Saito]\label{saito}
	Let $\ell:\,X\to W$ be a relative minimal elliptic surface without multiple fiber.
	Assume that $\chi(\calo_X)\geq 2$ and $\ell$ is non-isotrivial. Then the infinitesimal Torelli theorem holds for $X$, i.e.,
	the following infinitesimal period map is injective.
	\[
	\delta:\,H^1(X,T_X) \to \text{Hom}\big(H^{2,0}(X),H^{1,1}(X)\big) \oplus \text{Hom}\big(H^{1,1}(X),H^{0,2}(X)\big).
	\]
\end{thm}

Recall that $\ell$ being non-isotrivial means that the fibration $\ell$ cannot become trivial after any finite base change.
In our case, this is also equivalent to saying that
the functional invariant $J(X)$ is not constant as in \cite{saito-83}.
By the Serre duality, one obtains immediately the following.
\begin{lem}[{\cite[Lemma 3.1]{saito-83}}]
The period map $\delta$ is injective if and only if the cup-product map
\begin{equation}
  \label{cup-product}
  \mu:\, H^0(X,\Omega^2_X)\otimes H^1(X,\Omega^1_X) \to H^1(X,\Omega^1_X \otimes \Omega^2_X)
\end{equation}
is surjective.
\end{lem}

\begin{proof}[Proof of \autoref{saito}]
	It suffices to prove the map $\mu$ in \eqref{cup-product} is surjective.
    Since the elliptic fibration $\ell:\,X \to W$ is non-isotrivial, one has
    \[
    q(X) = g(W) + h^1(\ell_*\omega_{X/W} \otimes \omega_W) = g(W) + h^0\left( (\ell_*\omega_{X/W})^{\vee} \right) = g(W).
    \]

    If $g(W)=0$ and $\chi(\calo_X)=2$, then $q(X)= g(W)=0$ and $p_g(X)=\chi(\calo_X)+q(X)-1=1$. It follows that $X$ is a K3 surface,
	and hence \eqref{cup-product} is surjective. Therefore, we may assume that
	\begin{equation}\label{eqn-4-1}
		\chi(\calo_X)\geq 3,\qquad \text{if~}g(W)=0.
	\end{equation}

Since the surface we are concerned about admits an elliptic fibration $\ell:\, X \to W$, we can use the Leray spectral sequence to decompose (\ref{cup-product}).
\begin{equation*}
0 \to H^1(W,\ell_*\Omega^1_X) \to H^1(X,\Omega^1_X) \to H^0(W,R^1\ell_*\Omega^1_X) \to 0,
\end{equation*}
\begin{equation*}
0 \to H^1(W,\ell_*\Omega^1_X \otimes \Omega^2_X) \to H^1(X,\Omega^1_X \otimes \Omega^2_X) \to H^0(W,R^1\ell_*(\Omega^1_X \otimes \Omega^2_X)) \to 0.
\end{equation*}

Note that the cup-product $\mu$ is compatible with the Leray spectral sequence, see $(12.2.6)$ in E.G.A III \cite{ega-3}. Therefore, $\mu$ is surjective if and only if the following cup-products $\mu_1$, $\mu_2$ are surjective.
\begin{equation}
  \label{mu-1}
  \mu_1:\, H^0(W,\ell_*\Omega^2_X) \otimes H^1(W,\ell_*\Omega^1_X) \to H^1(W,\ell_*(\Omega^1_X \otimes \Omega^2_X)),
\end{equation}
\begin{equation}
  \label{mu-2}
  \mu_2:\, H^0(W,\ell_*\Omega^2_X) \otimes H^0(W,R^1\ell_*\Omega^1_X) \to H^0(W,R^1\ell_*(\Omega^1_X \otimes \Omega^2_X)).
\end{equation}

From {\cite[Lemma~4.2,~(4.10)]{saito-83}} we obtain the following short exact sequence
\begin{equation}
  \label{fund-ses}
0 \to \ell^*\Omega^1_W \to \Omega^1_X \to \mathcal{I} \otimes \Omega^1_{X/W} \to 0,
\end{equation}
in the case when $\ell$ has no multiple fiber. We refer to {\cite[(4.8)]{saito-83}} for the precise definition of the ideal sheaf $\mathcal{I}$ on $X$. Consider the long exact sequence associated with \eqref{fund-ses}
\[
0 \to \Omega^1_W \to \ell_*(\Omega^1_X) \xrightarrow{\alpha} \ell_*(\mathcal{I} \otimes \Omega^1_{X/W}) \xrightarrow{\delta} R^1\ell_*(\ell^*\Omega^1_W)=\Omega^1_W \otimes R^1\ell_*(\calo_X).
\]
Since the edge map $\delta$ factors through the Kodaira-Spencer map associated with the fibration $\ell$, the non-isotriviality of $\ell$ implies that $\delta \neq 0$. On the other hand, the sheaf $\ell_*(\mathcal{I} \otimes \Omega^1_{X/W})$ has generic rank one. Thus the image of $\alpha$ is torsion, which means that $\ell_*(\Omega^1_X)$ also has generic rank one.

Now we tensor \eqref{fund-ses} with $\Omega^2_X$ and produce the following long exact sequence:
\[
0 \to \Omega^1_W \otimes \ell_*\Omega^2_X \to \ell_*(\Omega^1_X \otimes \Omega^2_X) \xrightarrow{\beta} \ell_*(\mathcal{I} \otimes \Omega^1_{X/W}\otimes \Omega^2_X) \to \cdots.
\]
From the canonical bundle formula (cf. {\cite[Theorem~12.1]{bhpv}}) we notice that $\Omega^2_X= \ell^*\mathcal{L}$ for some invertible sheaf $\mathcal{L}$ on $W$. Then $\ell_*(\Omega^1_X \otimes \Omega^2_X) = \ell_*\Omega^1_X \otimes \mathcal{L}$, which also has generic rank one. This implies that $\beta$ has torsion image, and consequently one has
\[
H^1(W, \Omega^1_W \otimes \ell_*\Omega^2_X) \twoheadrightarrow H^1(W, \ell_*(\Omega^1_X \otimes \Omega^2_X)).
\]
By the Serre duality and canonical bundle formula, one has
\[
H^1(W, \Omega^1_W \otimes \ell_*\Omega^2_X) = H^0(W, (\ell_*\Omega^2_X)^{-1}) = H^0(W, \mathcal{L}^{\vee}),
\]
where $\mathcal{L}$ is a line bundle of degree $\chi(\calo_X)-2\chi(\calo_W)$ on $W$ (cf. {\cite[Corollary~12.3]{bhpv}}). Recall the assumption $\chi(\calo_X) \geq 2$ and \eqref{eqn-4-1}. Under these conditions $\mathcal{L}$ is a line bundle with positive degree on $W$. Thus $H^1(W, \Omega^1_W \otimes \ell_*\Omega^2_X) = H^0(W, \mathcal{L}^{\vee})=0$, which implies that $H^1(W, \ell_*(\Omega^1_X \otimes \Omega^2_X))=0$. Therefore, the cup-product $\mu_1$ in (\ref{mu-1}) is automatically surjective.

Now we analyze $R^1\ell_*\Omega^1_X$. Using the long exact sequence associated with \eqref{fund-ses} again, one obtains
\[
\ell_*(\mathcal{I} \otimes \Omega^1_{X/W}) \xrightarrow{\delta} \Omega^1_W \otimes R^1\ell_*(\calo_X) \to R^1\ell_*\Omega^1_X \to R^1\ell_*(\mathcal{I} \otimes \Omega^1_{X/W}) \to 0
\]
which implies
\begin{equation}
  \label{coker-ses}
0 \to \mathrm{Coker}(\delta) \to R^1\ell_*\Omega^1_X \to R^1\ell_*(\mathcal{I} \otimes \Omega^1_{X/W}) \to 0.
\end{equation}

As we discussed before, the non-isotriviality of $\ell$ implies that $\delta \neq 0$, and thus $\mathrm{Coker}(\delta)$ is torsion. On the other hand, using the canonical bundle formula one computes
\[
R^1\ell_*(\mathcal{I} \otimes \Omega^1_{X/W}) = R^1\ell_*(\mathcal{I} \otimes \omega_{X/W}) = R^1\ell_*\left(\mathcal{I} \otimes \ell^*(R^1\ell_*\calo_X)^{\vee}\right) = R^1\ell_*\mathcal{I} \otimes (R^1\ell_*\calo_X)^{\vee}.
\]
From {\cite[Proposition~4.3,~(ii)]{saito-83}}, we know that $R^1\ell_*\mathcal{I} = R^1\ell_*\calo_X \oplus T_1$, where $T_1$ is a torsion sheaf on $W$. Thus
\[
R^1\ell_*(\mathcal{I} \otimes \Omega^1_{X/W}) = \calo_W \oplus \textrm{some torsion sheaf}.
\]
By \eqref{coker-ses} $R^1\ell_*\Omega^1_X$ has the same form. Write
\[
R^1\ell_*\Omega^1_X = \calo_W \oplus T,
\]
where $T$ is the torsion part.

From the above discription of $R^1\ell_*\Omega^1_X$ one easily deduces the following isomorphisms:
\begin{equation}
  \label{isom-3}
 H^0(W, R^1\ell_*\Omega^1_X) \cong H^0(W,\calo_W) \oplus H^0(W,T),
\end{equation}
\begin{equation}
  \label{isom-4}
 H^0(W, R^1\ell_*(\Omega^1_X \otimes \Omega^2_X)) \cong H^0(W,\Omega^1_W \otimes \ell_*\omega_{X/W}) \oplus H^0(W,T \otimes \Omega^1_W \otimes \ell_*\omega_{X/W}).
\end{equation}

By (\ref{isom-3}) and (\ref{isom-4}), we can decompose the cup-product $\mu_2$ in (\ref{mu-2}) into two maps:
\begin{equation*}
  \label{mu-2-1}
  \mu^1_2:\, H^0(W,\Omega^1_W \otimes \ell_*\omega_{X/W}) \otimes H^0(W,\calo_W) \to H^0(W,\Omega^1_W \otimes \ell_*\omega_{X/W}),
\end{equation*}
\begin{equation*}
  \label{mu-2-2}
  \mu^2_2:\, H^0(W,\Omega^1_W \otimes \ell_*\omega_{X/W}) \otimes H^0(W,T) \to H^0(W,T \otimes \Omega^1_W \otimes \ell_*\omega_{X/W}).
\end{equation*}
Here we also use the canonical bundle formula to get $\ell_*\Omega^2_X= \Omega^1_W \otimes \ell_*\omega_{X/W}$.

The surjectivity of $\mu^1_2$ is obvious. And one can conclude immediately $\mu^2_2$ is surjective once we know that $\Omega^1_W \otimes \ell_*\omega_{X/W}$ has no base point, which follows from the next lemma. This completes the proof.
\end{proof}

\begin{lem}\label{base-point-free}
  If $\chi(\calo_X)\geq 2$ and $\ell$ is non-isotrivial, then $|\Omega^1_W \otimes \ell_*\omega_{X/W}|$ is base point free.
\end{lem}
\begin{proof}
First we note that
$$p_g(X) = h^0(X,\Omega^2_X)=h^0(W,\ell_*\Omega^2_X)=h^0(W,\Omega^1_W \otimes \ell_*\omega_{X/W}).$$
By the Riemann-Roch theorem,
\[
p_g(X)=h^1(W,\Omega^1_W \otimes \ell_*\omega_{X/W}) +(1-g(W)) + \text{deg}(\Omega^1_W \otimes \ell_*\omega_{X/W}).
\]
Since $\ell:\, X \to W$ is non-isotrivial,
it follows that $\ell_*\omega_{X/W}$ is a line bundle with positive degree on $W$, and thus
$$h^1(W,\Omega^1_W \otimes \ell_*\omega_{X/W})=h^0(W,(\ell_*\omega_{X/W})^{-1})=0.$$
By the assumption,
$$p_g(X) \geq q(X)+1=g(W) +1.$$ Hence
\[
\text{deg}(\Omega^1_W \otimes \ell_*\omega_{X/W}) \geq 2g(W).
\]
By \cite[Corollary\,3.3]{Hartshorne-77},
we conclude that $|\Omega^1_W \otimes \ell_*\omega_{X/W}|$ is base point free.
\end{proof}


\subsection{Relative Iitaka fibration and the proof of \autoref{main-thm2}}
In this section, we aim to prove \autoref{main-thm2}. First we observe that restricting the family of elliptic surfaces to any positive dimensional subvariety it still has maximal variation since the classifying map from $B$ to the coarse moduli space is quasi-finite. So after replacing the original family by this restriction, we only need to show that $\pi_1(B)$ is infinite and nonabelian.

The proof combines Saito's infinitesimal Torelli theorem and Simpson's nonabelian Hodge theory. It should be emphasized that we use not only the weight two period mapping associated to the family of elliptic surfaces, but also the weight one period mapping associated to the family of curves induced by the relative Iitaka fibration. Oguiso-Viehweg's analysis on the singular locus of the elliptic surfaces is also crucial to our argument.


\begin{proof}[Proof of \autoref{main-thm2}]
The factorization of $f:\,\calx \to B$ by the Iitaka fibration gives a family of curves $h:\,\calw \to B$:
$$\xymatrix{
	\calx \ar[rr]^-{\ell} \ar[dr]_-{f} && \calw \ar[dl]^-{h}\\
	&B&} $$
The family $h$ induces a classifying map,
which is denoted by $\varphi_1:\,B \to \calm_{g(\calw_b)}$.
If the image of $\varphi_1$ in $\calm_{g(\calw_b)}$ has positive dimension,
then the family $h$ gives a non-trivial weight-1 variation of Hodge structures
on $B$.
Hence $\pi_1(B)$ is infinite and non-abelian;
since otherwise,
there exists no non-trivial variation of Hodge structures on $B$ by the Simpson's correspondence,
which contradicts the non-triviality of the weight-1 variation of Hodge structures.
Thus one can concentrate on the case of $\text{dim}(\varphi_1(B))=0$. Since $B$ is connected, we know that $\varphi_1(B)$ is just a point. So after a finite \'etale base-change one can assume that $\calw \cong B \times \calw_b$.

Let $b\in B$ be a general point, and $\ell_b:\,\calx_b \to \calw_b$ the corresponding elliptic fibration.
For simplicity we may omit the parameter $b$,
and denote by $\ell:\,X \to W$ a general member in the family.
We will divide the proof into the following two cases.

{\bf \vspace{2mm}
	 Case I:~ the elliptic fibration $\ell$ is non-isotrivial}.
	
   Since $\ell:\,X \to W$ is non-isotrivial, we know that the direct image sheaf $f_*\omega_{X/W}$ has positive degree, and thus
   \[
   \begin{array}{rl}
     p_g(X) - q(X) & = \chi(f_*\omega_{X/W} \otimes \omega_W) - g(W) \\
                 & = \mathrm{deg}(f_*\omega_{X/W} \otimes \omega_W) + 1 -2g(W) \\
                 & = \mathrm{deg}(f_*\omega_{X/W}) -1 \geq 0.
   \end{array}
   \]
Thus one gets $\chi(\calo_X)=p_g(X)-q(X)+1>0$.
	If $g(W)\geq 1$, one can do the \'etale base change of $\calw$ freely so that $\chi(\calo_X) \geq 2$, since $\calw \cong B \times W$ with $g(W) \geq 1$;
	if $g(W)=0$, then $\chi(\calo_X)\geq 3$ since $X$ is of Kodaira dimension one.
	Thus Saito's infinitesimal Torelli \autoref{saito} holds.
	Let $$\varphi_2:\,B \to \cald/\Gamma$$
	be the period map induced by the family $f:\,\calx \to B$.
	Then $\dim \varphi_2(B)>0$; otherwise
	the weight-$2$ period map $\varphi_2$ is constant, which means that the infinitesimal period map $\delta=0$, contradicting \autoref{saito}.
	This shows $\pi_1(B)$ is infinite and non-abelian
	by Proposition \ref{prop-1-1}.
	

{\bf \vspace{2mm}
	Case II:~ the elliptic fibration $\ell$ is isotrivial}.

First note that the fundamental group $\pi_1(B)$ is infinite and non-abelian
if either $\varphi_1$ or $\varphi_2$ is a non-constant map.
So we can concentrate on the case that both $\varphi_1$ and $\varphi_2$ are constant.
To go further, we need a result due to Oguiso and Viehweg:
\begin{prop}[{\cite[Proposition 5.1]{ov-01}}] \label{sing-locus}
Let $f:\calx \to \calw \to B$ be a smooth projective family of minimal elliptic surfaces with $\chi(\calo_{\calx_b}) \geq 2$ and $\kappa(\calx_b)=1$ for all $b \in B$. Assume that the weight-2 period map $\varphi_2$ is constant. Let $D$ be the discriminant locus in $\calw$. Then $D$ is \'etale over $B$.
\end{prop}
Since there is some slight difference between our situation and Oguiso-Viehweg's (the multiple locus are nonempty in their case), we shall sketch the proof of the proposition above. First we recall the following lemma:
\begin{lem}[{\cite[Lemma 5.2]{ov-01}}] \label{root}
Let $\ell:\,X \to W$ be a minimal elliptic surface of Kodaira dimension one with $\chi(\calo_X) \geq 2$. Define the numbers $n_i,m_j$ and $l_k$ as the number of reducible fibres, according to the following list:
\[
\begin{array}{c|c|c|c|c|c|c|c}
\mathrm{type} & {}_mI_i & IV\, \mathrm{or} & III\, \mathrm{or} & I^*_j & II^* & III^* & IV^* \\
              & (m \geq 1,i \geq 4) & {}_mI_3 & {}_mI_2 &  &  &  & \\
\hline
\mathrm{number} &  &  &  &  &  &  &  \\
\mathrm{of \, fibres} & n_i & n_3 & n_2 & m_j & l_8 & l_7 & l_6 \\
\hline
\mathrm{Euler} &  &  &  &  &  &  &  \\
\mathrm{number} & i & 4\, \mathrm{or}\,3 & 3\, \mathrm{or}\,2  & j+6 & 10 & 9 & 8
\end{array}
\]
Let $N_b= \langle \alpha \in NS(\calx_b); (\alpha . F)=0 \, and \, (\alpha.\alpha)=-2 \rangle /\mathbb{Q} \cdot F \cap NS(\calx_b)$ be the root lattice. Here $F$ is a general fibre of the elliptic fibration $\ell_b:\, \calx_b \to \calw_b$. Then
\begin{itemize}
\item [i)] $N_b$ is generated by the classes of irreducible components of reducible fibres of $\ell_b$.
\item [ii)] The numbers $n_i,m_j$ and $l_k$ are uniquely determined by $N_b$ and by its decomposition
\[
N_b \simeq \bigoplus_{i \geq 2} A^{\oplus n_i}_i \oplus \bigoplus_{j \geq 0} D^{\oplus m_j}_{j+4} \oplus \bigoplus^8_{k=6} E^{\oplus l_k}_k
\]
in indecomposable sublattices.
\end{itemize}
\end{lem}

\bigskip
\begin{proof}[Proof of Proposotion\,\ref{sing-locus}]
By the Leftschetz $(1,1)$-Theorem,
$$NS(\calx_b) = H^{1,1}(\calx_b) \cap H^2(\calx_b,\mathbb{Z}).$$
Thus the triviality of the variations of Hodge structures forces
$NS(\calx_b)$ to be a constant local system over $B$.
So the root lattice $N_b$ in Lemma \ref{root}, as well as its decomposition, is independent of $b \in B$.
Since some different types of singular fibres may have the same Dynkin diagram, like type $IV$ and ${}_mI_3$, type $III$ and ${}_mI_2$ in the table, one needs to distinguish those special types (in our situation the multiplicity $m$ always equals to one).
Recall that the elliptic surface $\ell:\, X \to W$ we are concerned about is iso-trivial.
Since the semi-stable singular fibres are preserved under the base change,
we know that type $I_i$ singular fibres ($i \geq 2$) can never appear.
Therefore, the constantness of the root lattice $N_b$ is sufficient to guarantee that singular loci do \emph{not} collide.
\end{proof}

Come back to the proof of \autoref{main-thm2}.
After some \'etale base change we can assume that the discriminant locus of $\calx \to \calw$ are sections of $\calw \cong B \times W \to B$, denoted by $\{D_i\}$.
Each $D_i$ induces a morphism $$s_i:\, B \lra W.$$

\begin{lem}\label{nonconstant}
At least one of those $s_i$'s is non-constant.
\end{lem}
\begin{proof}
We argue by contradiction. Suppose all the sections are constant.
Then each elliptic surface $\ell_b:\, \calx_b \to \calw_b=W$ in our family admits the same discriminant locus $D$ over a fixed curve $W$. Denote by $W_0:= W \setminus D$. Then we have a smooth elliptic fibration $\calx \setminus \ell^{-1}(D \times B) \to W_0 \times B$, which descends to an elliptic fibration $\mathcal{E} \to B$.

If the elliptic fibration $\mathcal{E} \to B$ is non-isotrivial, then one can consider the associated non-trivial weight-1 variation of Hodge structures on $B$. By the Simpson's correspondence, we know that $\pi_1(B)$ is infinite and non-abelian.

So we will focus on the case that $\mathcal{E} \to B$ is also isotrivial. This implies that the elliptic fibration $\ell:\, \calx \to \calw=W \times B$ is isotrivial. Then there exists a finite cover $\calw' \to \calw$ ramified at the discriminant locus $D \times B$, such that $\calx \times_{\calw}\calw' \cong E \times \calw'$ for some elliptic curve $E$. Denote by $d$ the degree of the finite cover $\calw' \to \calw$.

Then each elliptic surface $\calx_b \to W$ can be written as a finite quotient of the product surface $E \times \calw'_b$, where $\calw'_b \to W$ is a finite cover of degree $d$ and ramified over $D$. By the Riemann Existence Theorem, the isomorphism classes of such covers $\calw'_b \to W$ are in an one-to-one correspondence with the conjugacy classes of $\pi_1(W_0) \to S_d$, which are finitely many.

Therefore, the family of base curves $\calw' \to B$ is isotrivial, and each elliptic surface $\calx_b$ is a finite quotient of the same product surface $E \times W'$, where $W'$ is the general fiber of $\calw' \to B$. Such a family $\calx \to B$ cannot be non-isotrivial.
\end{proof}



If $W = \mathbb{P}^1$, then from \cite{par-68,ara-71} we know that each elliptic surface $\mathcal{X}_b$ has at least three singular fibers since it is of Kodaira dimension one.
Take an isomorphism
$$\eta:\, B \times W \to B \times W, \qquad (b,w) \mapsto \left(b, \frac{w-s_1(b)}{w-s_2(b)} \cdot \frac{s_3(b)-s_2(b)}{s_3(b)-s_1(b)} \right).$$
Then one sees that $s_1$, $s_2$ and $s_3$ become constant; in fact,
$$s_1(B)=0,\quad s_2(B)=\infty,\quad s_3(B)=1.$$
According to Lemma~\ref{nonconstant}, there exists at least one section, denoted by $s_4$, such that $s_4$ is a non-constant morphism from the base $B$ to $W=\mathbb{P}^1$.
Moreover, since those sections $s_i$'s are disjoint,
it follows that $$s_4(B) \subseteq \mathbb{P}^1\setminus \{0,1,\infty\}.$$
This implies in particular that $\pi_1(B)$ is infinite and non-abelian as required.

If $g(W)=1$, one can take an isomorphism
$$\eta:\, B \times W \to B \times W, \qquad (b,w) \mapsto \left(b, w\cdot s_1(b)^{-1}\right),$$
where `$\cdot$' is the multiplication for the group structure on $W$.
Similar to the above case, one finds at least one section, denoted by $s_2$,
such that
$s_2$ is a non-constant morphism from $B$ to $W$ and that $s_2(B)\subseteq W\setminus \{s_1(B)=O\}$,
which implies also that $\pi_1(B)$ is infinite and non-abelian.

Finally if $g(W)\geq 2$,
then there exists at least one section, denoted by $s_1$,
such that $s_1$ is a non-constant morphism from $B$ to $W$.
Since $W$ is of genus $\geq 2$, $\pi_1(B)$ is infinite and non-abelian.
This completes the proof.
\qedhere
\end{proof}




\end{document}